\documentclass[reqno]{amsart}
\usepackage{amsthm,amsmath,amssymb,amsfonts}
\usepackage{epsfig}
\usepackage{graphicx,color}
\usepackage{pstricks}
\usepackage{verbatim}
\usepackage[T1]{fontenc}
\usepackage[ansinew]{inputenc}
\usepackage{mathrsfs}
\usepackage[backref,colorlinks]{hyperref}

\newtheorem{theorem}{Theorem}
\newtheorem{thm}[theorem] {Theorem}
\newtheorem{lemma}[theorem]{Lemma}

\newtheorem{claim}           [theorem] {Claim}

\newtheorem{conj}[theorem]{Conjecture}

\newcounter{caseq}[ruecksetzer]

\theoremstyle{remark}

\newtheorem{case}     [caseq]       {Case}

\let\eps=\varepsilon
\let\theta=\vartheta
\let\rho=\varrho
\let\sigma=\varsigma

\let\polishlcross=\l
\def\l{\ifmmode\ell\else\polishlcross\fi}

\newcommand{\vect}[1]{\mbox{\boldmath$#1$}}
\newcommand{\ex}{\mathrm{ex}}
\newcommand{\cycle}[3]{C^{(#1,#2)}_#3}
\newcommand{\hpath}[3]{P^{(#1,#2)}_{#3}}

\newcommand{\vecsx}{\boldsymbol{x}_{k-1}}
\newcommand{\vecsy}{\boldsymbol{y}_{k-1}}
\newcommand{\cA}{\mathcal{A}}
\newcommand{\cC}{\mathcal{C}}
\newcommand{\EE}{\ensuremath{\mathbb{E}}}

\begin{document}
\title[On Extremal Hypergraphs for Hamiltonian Cycles]{On Extremal Hypergraphs for Hamiltonian Cycles}

\author[R.~Glebov]{Roman Glebov}
\address{Institut f\"{u}r Mathematik, Freie Universit\"at Berlin, Arnimallee 3-5, D-14195 Berlin, Germany}
\email{glebov\,\textbar{}\,person@math.fu-berlin.de}
\thanks{The first author was supported by DFG within the research training group "Methods for Discrete Structures".}
\author[Y.~Person]{Yury Person}
\thanks{The second author was supported by GIF grant no.~I-889-182.6/2005.}
\author[W.~Weps]{Wilma Weps}

\date{\today}
\begin{abstract}
We study sufficient conditions for Hamiltonian cycles in hypergraphs, and obtain both Tur\'an- and Dirac-type results.
While the Tur\'an-type result gives an exact threshold for the appearance of a Hamiltonian cycle in a hypergraph depending only on the extremal number of a certain path, the Dirac-type result yields a sufficient condition relying solely on the minimum vertex degree.
\end{abstract}

\maketitle

\setcounter{footnote}{0}
\renewcommand{\thefootnote}{\fnsymbol{footnote}}

\section{Introduction and Results}

\subsection{ Tur\'an-type Results}

For a fixed graph $G$ and an integer $n$ the \emph{extremal number} $\ex \left( n, G \right)$ of $G$ is the largest integer $m$ such that there exists a graph on $n$ vertices with $m$ edges that does not contain a subgraph isomorphic to $G$. The corresponding graphs are called \emph{extremal graphs}. Naturally, one can extend this definition to a forbidden spanning structure, e.g.\ a Hamiltonian cycle (for definition see e.g.\ \cite{Bollobas}). In \cite{Ore} Ore proved that a non-Hamiltonian graph on $n$ vertices has at most $\binom{n-1}{2} + 1$ edges, and further, that the unique extremal example is given by an $(n-1)$-clique and a vertex of degree one that is adjacent to one vertex of the clique.

A \emph{$k$-uniform hypergraph $H$}, or $k$-graph for short, is a pair $(V,E)$ with a vertex set $V=V(H)$ and
an edge set $E = E(H) \subseteq \binom{V}{k}$. 
Since in this paper we always deal with $k$-graphs, and the usual $2$-uniform graphs have no special meaning for us, 
we also might use the siplified term {\em graph} for $k$-graphs.

There are several definitions of Hamiltonian cycles in hypergraphs, e.g.\ Berge Hamiltonian cycles \cite{Bermond}. This paper yet follows the definition of Hamiltonian cycles established by Katona and Kierstead \cite{KK} as it has become more and more popular in research.

An \emph{$l$-tight Hamiltonian cycle} in $H$, $0 \leq l \leq k-1$, $(k-l)\big||V(H)|$, is a spanning sub-$k$-graph whose vertices can be cyclically ordered in such a way that the edges are segments of that ordering and every two consecutive edges intersect in exactly $l$ vertices. 
More formally, it is a graph isomorphic to $([n],E)$ with \[E=\left\{\{i(k-l)+1, i(k-l)+2, \ldots, i(k-l)+k\}:~0\leq i<\frac{n}{k-l}\right\},\]
where addition is made modulo $n$.
We denote an $l$-tight Hamiltonian cycle in a $k$-graph $H$ on $n$ vertices by $\cycle{k}{l}{n}$, and call it $tight$ if it is $(k-1)$-tight.

Working on her thesis~\cite{Dana} in coding theory, Woitas raised the question whether removing $\binom{n-1}{2}-1$ edges from a complete $3$-uniform hypergraph on $n$ vertices leaves a hypergraph containing a $1$-tight Hamiltonian cycle. A generalization of this problem is to estimate the extremal number of  Hamiltonian cycles in $k$-graphs.

Katona and Kierstead were the first to study sufficient conditions for the appearance of a $\cycle{k}{k-1}{n}$ in $k$-graphs.
In \cite{KK} they showed that for all integers $k$ and $n$ with $2 \leq k$ and $2k-1 \leq n$,
\[ \ex(n,\cycle{k}{k-1}{n}) \geq \binom{n-1}{k} + \binom{n-2}{k-2}.\]
In the same paper Katona and Kierstead proved, that this bound is not tight for $k=3$ by showing that for all integers $n$ and $q$ with $q \geq 2$ and $n = 3q+1$,
\[ \ex(n,\cycle{3}{2}{n}) \geq \binom{n-1}{3} + n - 1.\]
In \cite{Tuza} Tuza gave a construction for general $k$ and tight Hamiltonian cycles, improving the lower bound to
\[ \ex(n,\cycle{k}{k-1}{n}) \geq \binom{n-1}{k} + \binom{n-1}{k-2},\]
if a Steiner system $S(k-2,2k-3,n-1)$ exists. Also for all $k,n$ and $p$ such that a partial Steiner system $PS(k-2,2k-3,n-1)$ of order $n-1$ with $p \binom{n-1}{k-2} / \binom{2k-3}{k-2}$
blocks exists, Tuza proved the bound
\[ \ex(n,\cycle{k}{k-1}{n}) \geq \binom{n-1}{k} + p \binom{n-1}{k-2}.\]

An intuitive approach to forbid Hamiltonian cycles in hypergraphs is to prohibit certain structures in the \emph{link} of one fixed vertex.
For a vertex $v \in V$, we define the \emph{link of $v$ in $H$} to be the $ \left(k-1 \right)$-graph $H(v)= \left(V\backslash \{v\},E_v \right)$ with $\{x_1, \ldots, x_{k-1}\} \in E_v$ iff $\{v, x_1, \ldots, x_{k-1}\} \in E(H)$.

The structure of interest in this case is a generalization of a path for hypergraphs.

An {\em $l$-tight $k$-uniform $t$-path}, denoted by $\hpath{k}{l}{t}$, is a $k$-graph on $t$ vertices, $(k~-~l)~|~(t~-~l)$, such that
there exists an ordering of the vertices, say $(x_1, \ldots, x_t)$,  in such a way that the edges are segments of that ordering
and every two consecutive edges intersect in exactly $l$ vertices.
 Observe that a $\hpath{k}{l}{t}$ has $\frac{t-l}{k-l}$ edges.
A $k$-uniform $(k-1)$-tight path is called \emph{tight}, and whenever we consider a path we assume it to be tight unless stated otherwise.

For arbitrary $k$ and $l$ we give the exact extremal number and the extremal graphs of $l$-tight Hamiltonian cycles in this paper. 
The extremal number and the extremal graphs rely on the extremal number of $P(k,l) := \hpath{k-1}{l-1}{\left \lfloor \frac{k}{k-l} \right \rfloor (k-l) + l -1}$, 
and its extremal graphs, respectively.

\begin{thm}
\label{l-tight}
For any $k \geq 2, l \in \{0, \ldots, k-1\}$ there exists an $n_0$ such that for any $n \geq n_0$ and $( k-l ) | n$,
\[ \ex \left(n, \cycle{k}{l}{n} \right) = \binom{n-1}{k} + \ex \left (n-1, P(k,l) \right)\]
holds.
Furthermore, any extremal graph on $n$ vertices contains an $(n-1)$-clique and a vertex whose link is $P(k,l)$-free.
\end{thm}
Notice, that $P(k,l)$ contains $\left\lfloor \frac{k}{k-l} \right\rfloor$  hyperedges.

For $k=3$ and $l=1$ Theorem~\ref{l-tight} answers the aforementioned question of Woitas~\cite{Dana} that indeed $\binom{n-1}{3}+1$ hyperedges ensure an
existence of a $1$-tight Hamiltionian cycle $\cycle{3}{1}{n}$ for $n$ large enough.

For $k=3$ and $l=2$ Theorem~\ref{l-tight} states that there exists an $n_0$ such that for any $n \geq n_0$,
\begin{align*}
\ex \left(n, \cycle{3}{2}{n} \right) & =  \binom{n-1}{3} + \ex \left (n-1, \hpath{2}{1}{4} \right) \\
 & = \begin{cases}
\binom{n-1}{ 3} + n-1, & 3 \, | \, n-1 \\
\binom{n-1}{3} + n-2 , & \text{otherwise.}
\end{cases}
\end{align*}

Note that this not only goes along with Katona and Kierstead's remark, but further specifies it for the special case $k=3$.\\

Actually, in this paper we prove a stronger statement, namely that with one
more edge we find a Hamiltonian cycle that is $l$-tight in the neighborhood of one
vertex and is $ \left(k-1 \right)$-tight on the rest.

Using the result by Gy\"ori, Katona, and Lemons \cite{GKL} stating that
\[(1 + o(1)) \binom{n-1}{k-2} \leq  \ex \left (n-1,\hpath{k-1}{k-2}{2k-2} \right) \leq (k-1) \binom{n-1}{k-2},\]
we obtain lower and upper bounds for $l=k-1$:
\[\binom{n-1}{k} + (1 + o(1)) \binom{n-1}{k-2} \leq  \ex \left (n,\cycle{k}{l}{n} \right)  \leq \binom{n-1}{k} + (k-1) \binom{n-1}{k-2}.\]
Note that the upper bound also holds for $l \neq k-1$.

In our proof we make use of the \emph{absorbing technique} that was originally developed by  R\"odl, Ruci\'nski and Szemer\'edi.

\subsection{Dirac-type Results}

The problem of finding Hamiltonian cycles and perfect matchings in $2$-graphs has been studied very intensively. There are plenty beautiful conditions guaranteeing the existence of such cycles, e.g.\ Dirac's condition \cite{Dir}.

Over the last couple of years several Dirac-type results in hypergraphs were shown, and along with them, different definitions of \emph{degree} in a $k$-graph were introduced.  They all can be captured by the following definition.
The \emph{degree} of $\{ x_1, \ldots, x_{i}\}$, $1 \leq i \leq k-1$, in a $k$-graph $H$ is the number of edges the set is contained in and is denoted by $\deg( x_1, \ldots, x_{i})$.
Let
\[ \delta_d (H) := \min \{ \deg (x_1, \ldots, x_d) | \{ x_1, \ldots, x_d \} \subset V(H) \} \]
for $0 \leq d \leq k-1$. If the graph is clear from the context, we omit $H$ and write for short $\delta_d$. Note that $\delta_0 = e(H) := |E(H)|$ and $\delta_1$ is the minimum vertex degree in $H$.

Following the definitions of  R\"odl and Ruci\'nski in \cite{Survey}, denote for every $d,k,l$ and $n$ with $0 \leq d \leq k-1$ and $(k-l) | n$ the number $h^l_d(k,n)$ to be the smallest integer $h$ such that every $n$-vertex $k$-graph $H$ satisfying $\delta_d(H) \geq h$ contains an $l$-tight Hamiltonian cycle.
Observe that $h^l_0 (k,n) = \ex \left(n, \cycle{k}{l}{n} \right) + 1$.

In \cite{KK} Katona and Kierstead showed that $h^{k-1}_{k-1}(k,n) \geq \left \lfloor \frac{n-k+3}{2} \right \rfloor$ by giving an extremal construction. Their implicit conjecture that this bound is tight was confirmed for $k=3$ by R\"odl, Ruci\'nski and Szemer\'edi in \cite{RRSdirac} asymptotically and in \cite{RRScond} exactly.
For $k \geq 4$ the same authors showed in \cite{RRSapprox} that $ h^{k-1}_{k-1}(k,n)\sim \frac{1}{2}n$.
Generalizing the results to other tightnesses, Markstr\"om and Ruci\'nski proved in \cite{MR} that $ h^{l}_{k-1}(k,n)\sim \frac{1}{2}n$ if $(k-l) | k,n$.
In \cite{KMO} K\"uhn, Mycroft and Osthus proved that
\[ h^{l}_{k-1}(k,n)\sim \frac{n}{\left \lceil \frac{k}{k-l} \right \rceil (k-l)} \]
if $k-l$ does not divide $k$ and $(k-l)|n$, proving a conjecture by H\`an and Schacht \cite{HS}.
For further information, an excellent survey of the recent results can be found in \cite{Survey}.

R\"odl and Ruci\'nski conjectured in \cite{Survey} that for all $1 \leq d \leq k-1$, $k | n$,
\[ h^{k-1}_d (k,n) \sim h^0_d (k,n).\]
Further notice that $0$-tight Hamiltonian cycles $\cycle{k}{0}{n}$ are perfect matchings covering all vertices. A perfect matching may be considered the ``simplest''
 spanning structure and  there are several results about $ h^0_{d} (k,n)$, see for example~\cite{RRSpm, HPS, KuhOstTre}.

Noting the fact that there are virtually no results on $h^l_d(k,n)$ for $d \leq k-2$, R\"odl and Ruci\'nski remarked in \cite{Survey} that it does not even seem completely trivial to show $h^2_1(3,n) \leq c \binom{n-1}{2}$ for some constant $c < 1$. Further, they gave the following bounds
\[ \left( \frac{5}{9} + o(1) \right) \binom{n-1}{2} \leq h^2_1(3,n) \leq \left( \frac{11}{12} + o(1) \right) \binom{n-1}{2}.\]

We show the following upper bound on $h^{k-1}_1 (k,n)$.

\begin{thm}
\label{delta1}
For any $k \in \mathbb{N}$ there exists an $n_0$ such that every $k$-graph $H$ on $n \geq n_0$ vertices with $\delta_1 \geq \left( 1- \frac{1}{22 \left( 1280 k^3\right)^{k-1}} \right) \binom{n-1}{k-1}$ contains a tight Hamiltonian cycle.
\end{thm}

Note that Theorem~\ref{delta1} implies
\[ h^l_d (k,n) \leq \left( 1- \frac{1}{22 \left( 1280 k^3\right)^{k-1}} \right) \binom{n-d}{k-d} \]
for all $l \in \{ 0, \ldots, k-1 \}$ and all $1 \leq d \leq k-1$. This shows that there exists a constant $c < 1$ such that for all $l,d$
\[ h^l_d (k,n) \leq c \binom{n-d}{k-d} \]
holds, although this constant is clearly far from being optimal.

\section{Proofs}

\subsection{Outline of the Proofs}

In the following we give a brief overview over the structure of the proofs of Theorems~\ref{l-tight} and~\ref{delta1}.
For this section we define for every $k \in \mathbb{N}$
\[ \eps = \frac{1}{22 (1280k^3)^{k-1}} \]
and
\[ \rho = \left(22 \eps \right)^{\frac{1}{k-1}}. \]

Suppose $H=(V,E)$ is a $k$-graph on $n$ vertices with $\delta_1 \geq (1-\eps)\binom{n-1}{k-1}$ and $n$ sufficiently large.
By an $end$ of a path $\hpath{k}{l}{t}$ we mean the tuple consisting of its first $k-1$ vertices, $\left(x_1, \ldots, x_{k-1}\right)$, or the tuple consisting of its last $k-1$ vertices in reverse order, $\left(x_t, \ldots, x_{t-k+2} \right)$, considering the ordered vertices. For an $i$-tuple $(x_1, \ldots, x_i)$ in $H$ we write \vect{x_i}, $1 \leq i \leq n$.
We call $\vect{x_{k-1}}$ \emph{good} if all $x_i$s are pairwise distinct and for all $i \in \{1, \ldots, k-1\}$ it holds that
\begin{equation}\label{eq:good_tupl_cond}
 \deg (x_1, \ldots, x_{i} ) \geq \left(1 - \rho^{k-i} \right) \binom{n-i}{k-i}.
\end{equation}
A path is called \emph{good} if both of its ends are good.\\

\textbf{Outline of the proofs and some definitions:}

\begin{enumerate}
\item At first, we prove the existence of one $l$-tight good path
or several vertex-disjoint good tight paths containing the vertices of small degree, see Claim~\ref{match}. (Note that we do not need this step in the proof of Theorem~\ref{delta1}.)
\item We say that a tuple $\vect{x_{2k-2}}$ \emph{absorbs} a vertex $v\in V$ if both $\vect{x_{2k-2}}$ and
$ \left(x_1,\ldots, x_{k-1}, v, x_k, \ldots, x_{2k-2} \right)$ induce good paths in $H$, meaning that the corresponding ordering of the paths is $\vect{x_{2k-2}}$ or $ \left(x_1,\ldots, x_{k-1}, v, x_k, \ldots, x_{2k-2} \right)$, respectively, and the ends are good. Lemma~\ref{abs} ensures a set $\mathcal{A}$,
such that any remaining vertex can be absorbed by many tuples of $\mathcal{A}$. We call an element of $\mathcal{A}$ an \emph{absorber}.
\item For $\vect{x_i},\vect{y_j} \in V^{k-1}$ we define
\[ \vect{x_i} \Diamond \vect{y_j} := (x_1, \ldots, x_i, y_1, \ldots, y_j).\]
Let $\vect{x_{k-1}}$ and $\vect{y_{k-1}}$ be good. We say that a tuple $\vect{z_{k-1}}$ \emph{connects} $\vect{x_{k-1}}$ with $\vect{y_{k-1}}$ if $\left ( x_{k-1}, \ldots, x_1 \right) \Diamond \vect{z_{k-1}} \Diamond \vect{y_{k-1}}$ induces a path in $H$ with respect to the order. Notice that the connecting-operation is not symmetric. Lemma~\ref{con} guarantees a set $\mathcal{C}$ such that any pair of $(k-1)$-tuples in $H$ can be connected by many elements of $\mathcal{C}$. We call the elements of $\mathcal{C}$ \emph{connectors}.
\item We modify $\mathcal{A}$ and $\mathcal{C}$ such that $\mathcal{A}$, $\mathcal{C}$ and the element(s) of Step $1$ are pairwise vertex-disjoint.
\item In Lemma~\ref{onepath} we create a good tight path that contains all elements of the modified $\mathcal{A}$, respecting their ordering.
\item Using Lemma~\ref{extend}, we extend the path from Step $5$ until it covers almost all of the remaining vertices that neither participate in ($l$-tight or tight) good paths of Step $1$ nor in the modified $\mathcal{C}$.
\item Using connectors, we create a cycle containing the
($l$-tight or tight) good paths from Step $1$ and the good path from Step $6$.
\item In the final step all remaining vertices are absorbed by the absorbers in the cycle.
\end{enumerate}

\subsection{Auxiliary Lemmas}
In this part we derive the main tools used to prove Theorems~\ref{l-tight} and~\ref{delta1}.
For this subsection let $H=(V,E)$ be a $k$-graph on $n$ vertices with
\begin{equation}\label{eq:mindegree}
\delta_1 \geq (1-\eps)\binom{n-1}{k-1}.
\end{equation}
Recall $\eps = \frac{1}{22 (1280k^3)^{k-1}}$ and $n$ sufficiently large.\\

The following lemma provides us with an essential tool which we use to prove other statements in this subsection.

\begin{lemma}
\label{randomgood}
Let $\vect{x_{2k-2}}$ be chosen u.a.r.\ from $V^{2k-2}$. The probability that all $x_i$s are pairwise distinct and both $\left( x_1, \ldots, x_{k-1}\right)$ and $\left( x_{2k-2}, \ldots, x_{k}\right)$ are good is at least $ \frac{8}{11}$.
\end{lemma}
\begin{proof}
Let $a$ be the number of $(k-1)$-tuples that are not good and have $k-1$ distinct entries, i.e. that are taken from $V$ without repetition. 
Further, let $b_j$ be the number of $j$-tuples $\vect{y_j}$ with
$\deg(y_1, \ldots, y_j)~<~(1~-~\rho^{k-j})~\binom{n-j}{k-j}$, $j \in \{ 1, \ldots, k-1\}$, and all $y_j$s are again pairwise distinct. Thus, by the definition of a good tuple, for each tuple $\vect{y_{k-1}}$ that is not good and has pairwise distinct entries, there exists a $j$ such that $\vect{y_j}$ is one of the $b_{j}$ tuples with small degree.
Furthermore, for every $\vect{y_j}$ there are at most $\frac{(n-j)!}{(n-k+1)!}$ different $(k-1)$-tuples $\left( y_1, \ldots, y_j, z_1, \ldots, z_{k-1-j} \right)$ with pairwise distinct $z_j \in V \backslash \{y_1, \ldots, y_j\}$.
Hence,
\[a\leq \sum_{j=1}^{k-1}\frac{(n-j)!}{(n-k+1)!} b_j.\]

The second time we apply double counting, we recall that $H$ has at most $\eps \binom{n}{k}$ non-edges.
Each of the $b_j$ $j$-tuples is by definition in at least $\rho^{k-j} \binom{n-j}{k-j}$ non-edges,
and from every non-edge one obtains $\binom{k}{j} j!$ different $j$-tuples.
Thus,
\[\rho^{k-j} \binom{n-j}{k-j} b_j\leq \binom{k}{j} j! \eps \binom{n}{k}.\]

Putting the two bounds together, we obtain for a $vect{w_{k-1}}$ chosen u.a.r.\ from $V^{k-1}$
\begin{align*}
\Pr & \left[ \vect{w_{k-1}} \text{ is not good and has pairwise distinct entries} \right] =  \frac{a}{n^{k-1}} \\
& \leq \sum_{j=1}^{k-1}\frac{(n-j)!}{(n-k+1)!} b_j\frac{1}{n^{k-1}} \\
& \leq \sum_{j=1}^{k-1}\frac{(n-j)!}{(n-k+1)!} \frac{\binom{k}{j} j! \eps \binom{n}{k}}{\rho^{k-j} \binom{n-j}{k-j}}\frac{1}{n^{k-1}} \\
& \le\eps\sum^{k-1}_{i=1}{\frac{1}{\rho^{k-i}}}
< \frac{2 \eps}{\rho^{k-1}}
= \frac{1}{11}.
\end{align*}
Then for a $vect{x_{2k-2}}$ chosen u.a.r.\ from $V^{2k-2}$
\begin{align*}
 \Pr & [ \left( x_1, \ldots, x_{k-1} \right) \text{ and } \left( x_{2k-2}, \ldots, x_k \right) \text{ are good} ] \\
 & \geq  \frac{n(n-1) \ldots (n - 2k + 3)}{n^{2k-2}} - \frac{2}{11}
 \geq  1 - \frac{3}{11} = \frac{8}{11}.
\end{align*}
\end{proof}

For a given set $\mathcal{X}$ of tuples or graphs, we write $X$ when considering the corresponding vertex set.

\begin{lemma}
\label{abs}
For all $\gamma$, $0 < \gamma \leq \frac{1}{64k^2}$, there exists a set $\mathcal{A}$ of size at most $2 \gamma n$ consisting of disjoint $(2k-2)$-tuples, each inducing a good path with respect to its order, such that for each vertex $v\in V$ at least $\frac{\gamma n}{4}$ tuples in $\mathcal{A}$ absorb $v$.
\end{lemma}
\begin{proof}
By Lemma~\ref{randomgood}, we know that there are at least $\frac{8}{11}n^{2k-2}$ tuples $\vect{x_{2k-2}} \in V^{2k-2}$, such that
the $x_i$s are pairwise distinct and both $ \left(x_1, \ldots, x_{k-1} \right)$ and $ \left(x_{2k-2}, \ldots, x_{k} \right)$ are good. We denote the set of such tuples by $\cA'$.

Let $v$ be a vertex from $V$ and denote by $\cA_v$ the set of tuples $\vect{x_{2k-2}}$ from $\cA'$ such that, in addition,
\begin{itemize}
\item $ \left\{x_j, \ldots, x_{j+k-1} \right\} \in E(H)$, $1 \leq j \leq k-1$, and
\item $ \left\{v, x_j, \ldots, x_{j+k-2} \right\} \in E(H)$, $1 \leq j \leq k$.
\end{itemize}
Therefore, the set $\cA_v$ consists of those tuples that can absorb the vertex $v$.
From the minimum degree condition on $H$, see~\eqref{eq:mindegree}, it follows that
\begin{equation}\label{eq:cAv}
\left|\cA_v\right|\ge \frac{8}{11}n^{2k-2}- 2k\eps n^{2k-2}\ge \frac{7}{11}n^{2k-2}.
\end{equation}

Fix $\gamma$ with $0 < \gamma \leq \frac{1}{64k^2}$. Let $\cA$ be the set obtained by choosing each $ \left(2k-2 \right)$-tuple $\vect{x_{2k-2}} \in V^{2k-2}$ from $\cA'$  independently with probability $\frac{\gamma}{n^{2k-3}}$.

The expected size of $| \cA |$ is at most $\gamma n$ and we apply Chernoff's inequality (e.g.\ see \cite{AlonSpenc}):
\begin{equation}
\Pr \left[ |\cA|-\gamma n >  \gamma n \right] < e^{-\gamma n}. \label{Che1}
\end{equation}
This way, with high probability we obtain at most $2 \gamma n$ many $ \left(2k-2 \right)$-tuples.

Let $Y$ be the random variable taking the value $1$ whenever a pair of tuples in $\cA$ is not vertex-disjoint and the value $0$ else. Thus,
\begin{equation}
\mathbb{E}[Y] \leq (2k-2)^2n^{4k-5}\frac{\gamma^2}{n^{4k-6}} \leq 4 k^2 \gamma^2 n. \label{ExMar}
\end{equation}
Applying Markov's inequality, we obtain:
\begin{equation}
\Pr \left[Y > 8 k^2 \gamma^2 n\right] < \frac{1}{2}.\label{Mar}
\end{equation}

From~\eqref{eq:cAv} we infer by Chernoff's inequality that
\begin{equation}
\Pr \left[ |\cA_v\cap \cA| <  \frac{\gamma n}{2} \right] < e^{- \frac{1}{100} \gamma n} \label{Che2}
\end{equation}
since $\EE[|\mathcal{A'}_v\cap \cA|]\ge \frac{7}{11}\gamma n$.

With~\eqref{eq:cAv}, \eqref{Mar} and \eqref{Che2}, we see that if we delete from $\cA$ those pairs of tuples that have vertices in common, we obtain
with probability at least
\[
1/2- e^{-\gamma n}- n e^{- \frac{1}{100} \gamma n}>1/4
\]
 a new set satisfying the conditions of the lemma (note that $\frac{\gamma n}{2}-8k^2\gamma^2n\ge \frac{\gamma n}{4}$).
\end{proof}

The following lemma provides us with the essential tool to close the cycle.

\begin{lemma}
\label{con}
For all $\beta$, $0 < \beta \leq \frac{1}{64k^2}$, there exists a set $\mathcal{C}$ of size at most $2 \beta n$ consisting of pairwise disjoint $\left( k-1 \right)$-tuples, such that for each pair of good $(k-1)$-tuples there exist at least $\frac{\beta n}{4}$ elements in $\mathcal{C}$ that connect this pair.
\end{lemma}

\begin{proof}
For two good vertex-disjoint tuples $\vect{x_{k-1}}, \vect{y_{k-1}}$ in $H$, let $\cC_{\vecsx,\vecsy}$ be the set of all connectors that connect $\vecsx$ with $\vecsy$ and are vertex-disjoint from $\vecsx,\vecsy$. Recall that the following conditions hold for $z\in\cC_{\vecsx,\vecsy}$:
\begin{itemize}
\item $\{x_{k-i}, \ldots, x_1, z_1, \ldots, z_i\} \in E(H)$ for $ i \in \{ 1, \ldots, k-1 \}$, and
\item $\{z_{i-k+1}, \ldots, z_{k-1}, y_1, \ldots, y_{i-k+1}\} \in E(H)$ for $ i \in \{k, \ldots, 2k-2\}$.
\end{itemize}
From the condition~\eqref{eq:good_tupl_cond}, the definition of good tuples, and from~\eqref{eq:mindegree}, the minimum degree of $H$,
we infer
\[
|\cC_{\vecsx,\vecsy}|\ge \left(1-2\sum_{i=1}^{k-1} \rho^i\right)n^{k-1}-o(n^{k-1})\ge (1-4\rho k)n^{k-1}.
\]

Now, take  $\beta$ as asserted by the lemma and let $\cC':=\bigcup \cC_{\vecsx,\vecsy}$, where the union is over all vertex-disjoint good $(k-1)$-tuples $\vect{x_{k-1}}$ and $\vect{y_{k-1}}$. Define $\cC$ to be the set obtained by choosing each $\vect{z_{k-1}} \in \cC'$ independently with probability $\frac{\beta}{n^{k-2}}$.

Similarly to \eqref{Che1}, by
Chernoff's inequality:
\begin{equation*}
\Pr \left[ |\cC|-\beta n >  \beta n \right] < e^{-\beta n}.
\end{equation*}
With probability at least $\frac{1}{2}$ at most $4 k^2 \beta^2 n\le \beta n/4$ of the $(k-1)$-tuples have to be removed from $\cC$ to obtain a set of vertex-disjoint tuples, analogously to~\eqref{Mar}.

Analogously to \eqref{Che2}, for two good vertex-disjoint tuples $\vect{x_{k-1}}, \vect{y_{k-1}}$ in $H$,
\[ \Pr \left[ |\cC_{\vect{x_{k-1}},\vect{y_{k-1}}}\cap\cC| < \frac{\beta n}{2} \right] < e^{- \frac{1}{16} \gamma n}.\]

Therefore, we deduce with positive probability that after removing from $\cC$ all tuples that are not vertex-disjoint, we are left with a set that satisfies the conditions in the lemma.
\end{proof}

The next lemma helps us to connect a linear amount of small paths into a single path avoiding a small forbidden vertex subset.
\begin{lemma}
\label{onepath}
For any set $\mathcal{X}$ of vertex-disjoint $(2k-2)$-tuples that each induce a good path in $H$, $ |\mathcal{X}| \leq \frac{1}{4 k^2} n$, and any forbidden set $F \subset V$ of size at most $\frac{1}{8 k}n$, there exists a path $P$ containing all tuples of $\mathcal{X}$, respecting their individual ordering, such that $\left( V(P)\backslash X \right) \cap F= \emptyset$.
\end{lemma}

\begin{proof}
For arbitrary $\vect{x_{2k-2}}, \vect{y_{2k-2}} \in \mathcal{X}$, we choose a $\vect{z_{k-1}} \in V^{k-1}$ uniformly at random and define the events
\[ E^1 = \{ \vect{x_{2k-2}} \Diamond \vect{z_{k-1}} \text{ induces a path, respecting the ordering} \} \]
and
\[ E^2 = \{ \vect{z_{k-1}} \Diamond \vect{y_{2k-2}} \text{ induces a path, respecting the ordering} \}. \]
With $E^2_i$ being the event that $ \left\{z_{i}, \ldots, z_{k-1}, y_1, \ldots, y_{i} \right\} \in E$, $i \in \{ 1, \ldots, k-1\}$, we obtain
that
\[\Pr \left[E^2_i \right] \geq 1 - \rho^{k-i} - o(1),\]
since $\vect{y_{2k-2}}$ induces a good path, and the probability that at least two of the $k$ vertices coincide is $o(1)$.
Therefore,
\begin{multline*}
\Pr \left[E^2 \right]  \geq  1 - \sum^{k-1}_{i=1}{(1-\Pr \left[E^2_i \right])}
 \geq  1 - \sum^{k-1}_{i=1}{\left( \rho^{k-i} + o(1) \right)}
 \geq  1 - 2 \rho.
\end{multline*}
The same holds for $E^1$. Hence, by the union bound
\[\Pr \left[ E^1 \cap E^2 \right] \geq 1 - 4 \rho.\]
We choose an arbitrary ordering of $\mathcal{X}$. Iteratively, we consider two consecutive elements $\vect{x_{2k-2}}, \vect{y_{2k-2}}$ of $\mathcal{X}$.
The probability that a u.a.r.\ chosen $\vect{z_{k-1}} \in V^{k-1}$ connects $\vect{x_{2k-2}}$ with $\vect{y_{2k-2}}$
(meaning that both $E^1$ and $E^2$ hold) is at least $1 - 4 \rho$
and the probability that it is not vertex-disjoint to an already chosen element (connecting previous pairs of elements of $\mathcal{X}$), to $X$ or to $F$
is at most
\[ \left (\frac{k}{4 k^2} + \frac{2k}{4 k^2} + \frac{1}{8k} \right)k = \frac{7}{8} < 1 - 4 \rho\]
by the union bound.

Thus, we choose a  $\vect{z_{k-1}} \in V^{k-1}$ satisfying the conditions in the lemma, and iterate.
\end{proof}

The next lemma helps us find an almost spanning path in the hypergraph $H$.
\begin{lemma}
\label{extend}
For every good path $P$ and every set $F \subset V$ of size at most $k \rho n$, there exists a good path $P'$ that contains $P$ and covers all vertices except those from $F$ and at most $k \rho n$ further vertices.
\end{lemma}

\begin{proof}
Consider the longest good path $P'$ that contains $P$ and suppose that $|V(P')~\cup~F|~<~n~-~k\rho n$. Then choose one end $\vecsx$ of $P'$. 
Note that from~\eqref{eq:good_tupl_cond}, 
i.e.\, from the condition $ \deg \left(x_1, \ldots, x_{i} \right) \geq \left(1 - \rho^{k-i} \right) \binom{n-i}{k-i}$ for every $i$, 
it follows that, for every $i$, the number of vertices $v\in V(H)$ such that
\begin{equation}\label{eq:extend}
\deg(v,x_1,\ldots,x_i)\ge (1-\rho^{k-i-1})\binom{n-i-1}{k-i-1}
\end{equation}
is at least $n - \rho n$, implying that  $|V(P') \cup F| \geq n- k\rho n$.

Indeed, suppose for contradiction that there exists an $i$, such that the number of $v$s satisfying~\eqref{eq:extend} is less than $n-\rho n$. Then,
\begin{align*}
\deg(x_1,\ldots,x_i) & = \sum_{v\in V\setminus\{x_1,\ldots,x_i\}} \frac{1}{k-i}\deg(v,x_1,\ldots,x_i) \\
& < \frac{(n-\rho n)}{k-i}\binom{n-i-1}{k-i-1}+\frac{\rho n-i}{k-i} (1-\rho^{k-i-1}) \binom{n-i-1}{k-i-1} \\
& \leq \left(1 - \rho^{k-i} \right) \binom{n-i}{k-i},
\end{align*}
contradicting~\eqref{eq:good_tupl_cond}.
\end{proof}

\section{Proofs of Theorem~\ref{l-tight} and Theorem~\ref{delta1} }

\begin{proof}[Proof of Theorem~\ref{l-tight}]
Suppose $H= \left(V,E \right)$ is a $k$-graph on $n$ vertices, $n$ sufficiently large, with at least
\[\binom{n-1}{k} + \ex \left(n-1, P(k,l) \right)\]
edges and no vertex with a $P(k,l)$-free link.
Then the vertex set can be partitioned into two sets $V = V' \cup V"$ with $|V'|=n'$
and $V'' = \{v_1, \ldots, v_t\}$
such that
\begin{equation}\label{eq:delta_cond}
\delta_1 \left(H' \right)\geq \left(1 - \eps \right)\binom{n'}{k-1}
\end{equation}
 with $H' = H[V']$ and $\eps = \frac{1}{22(1280k^3)^{k-1}}$. To obtain $V'$, we iteratively delete vertices $v_1, \ldots, v_t$ of minimum degree from $H$ till the $\delta_1$-condition~\eqref{eq:delta_cond} holds. Counting the non-edges one observes that $t \leq \frac{2}{\eps}$.

The following claim provides an embedding of the vertices of $V''$.

\begin{claim}
\label{match}
There exists a set $\mathcal{S}$ of $t$ paths of type $\hpath{k}{k-1}{2k-2}$, if $t \geq 2$, or of type
$P(k,l)$,
if $t=1$, such that $v_i$ is in each edge of the $i^{th}$ element of  $\mathcal{S}$, $1 \leq i \leq t$, for some ordering of $\mathcal{S}$.
\end{claim}

We apply Lemma~\ref{abs} for $\gamma = \frac{1}{64 k^2}$ and Lemma~\ref{con} for $\beta = \frac{1}{ 1280 k^3}$ to $H'$.
As we want disjoint sets $A$, $C$, and $S$, 
we delete all elements from $\mathcal{A}$ that are not vertex-disjoint to an element from $\mathcal{C}\cup \mathcal{S}$. 
Thus, we delete at most $2k \beta n' + \frac{4k}{\eps} \leq \frac{\gamma n'}{20}$ absorbers overall, 
and  for every vertex $v \in V'$ there are at least $\frac{\gamma n'}{5}$ elements in the new set $\mathcal{A}$ absorbing $v$. 
Similarly, we make $C$ disjoint from $S$ still keeping at least $\frac{\beta n'}{5}$ connectors in $\mathcal{C}$ for each pair of good $(k-1)$-tuples.

Applying Lemma~\ref{onepath} on the new set $\mathcal{A}$, we obtain a good tight path in $H'$ containing all elements of $\mathcal{A}$ and no vertex from $C \cup S$.
We extend this path to one good path with Lemma~\ref{extend} such that it covers all but $k \rho n'$ vertices from $V'\backslash (C \cup S)$ and does not contain any vertex from $C \cup S$.

As a next step, we use connectors from $\mathcal{C}$ to connect the elements of $\mathcal{S}$ and the extended path to one cycle. This cycle absorbs the remaining vertices including the unused connectors, since $2 \beta n' + k \rho n' < \frac{\gamma n'}{5}$.
If $\mathcal{S}$ contains only one element, we obtain a Hamiltonian cycle that is tight except for the $l$-tight path from $\mathcal{S}$.
Otherwise, we obtain a tight Hamiltonian cycle.
Hence, there exists an $l$-tight Hamiltonian cycle.
\end{proof}
Note that we actually prove the bound for Hamiltonian cycles that are tight except in the link of at most one vertex.

Now deliver the missing proof of the above claim.
\begin{proof}[Proof of the Claim~\ref{match}]
We consider two cases.
\setcounter{caseq}{0}
\begin{case}[$\deg \left(v_1 \right) < \frac{\eps}{2} \binom{n-1}{k-1}$]
 In this case, the number of
missing edges yields a sufficient minimum degree in $H - v_1$, hence $t=1$.

Let $a$ be the number of hyperedges $\{x_1, \ldots, x_{k-1}\}$ in $H(v_1)$ that contain a subset $\{x_1, \ldots, x_j\}$, $j \in \{1, \ldots, k-1\}$,
satisfying
$\deg_{H'}(x_1, \ldots, x_j) < (1 - \rho ^{k-j}) \binom{n-1-j}{k-j},$
i.e.\, if there is a tuple $(x_1, \ldots, x_{k-1})$ obtained by an ordering of the edge that is not good in $H'$.
We denote the number of such $j$-sets by $b_j$ and observe that each of them lies in at most $\binom{n-1-j}{k-1-j}$ edges in $H(v_1)$.
Hence, we obtain
\[a\leq \sum_{j=1}^{k-1} \binom{n-1-j}{k-1-j} b_j.\]
We further call those $a$ edges \emph{bad}.

The second time we apply double counting, we set $c$ to be the number of non-edges in $H'$.
By definition, each of the $b_j$ $j$-sets lies in at least $\rho^{k-j} \binom{n-1-j}{k-j}$ non-edges of $H'$.
Note that each non-edge of $H'$ has exactly $2^k$ subsets. Henceforth,
\[\sum_{j=1}^{k-1}\rho^{k-j} \binom{n-1-j}{k-j} b_j \leq 2^kc.\]

Combining the two bounds, there are at most
\begin{align*}
a & \leq \sum_{j=1}^{k-1} \binom{n-1-j}{k-1-j} b_j = \sum_{j=1}^{k-1}\rho^{k-j} \binom{n-1-j}{k-j} b_j \frac{k-j}{n-k}\rho^{j-k} \\
& \leq \frac{k-1}{n-k} \rho^{1-k} \sum_{j=1}^{k-1} \rho^{k-j} \binom{n-1-j}{k-j} b_j \\
& \leq \frac{k-1}{n-k} \rho^{1-k} 2^k c \leq c
\end{align*}
edges in $H(v_1)$, which have an ordering producing a tuple that is not good in $H'$.
Observe that equality can only be obtained with $c=0$.

For $c=0$, there exists a $P(k,l)$ in the link of $v_1$ by assumption on $H$, and this path is good, hence, we are done. For $c>0$, we obtain $\deg(v_1) > c + \ex(n-1, P(k,l))$. We disregard bad hyperedges in $H(v_1)$ and using $a<c$, we still find a $P(k,l)$ in the link of $v_1$. The obtained path is good, proving the claim.
\end{case}

\begin{case}[$\deg \left(v_1 \right) \geq \frac{\eps}{2} \binom{n-1}{k-1} $]
In this case, we have for all $1\leq i \leq t$, $\deg \left(v_i \right) \geq \frac{1}{3} \binom{n-1}{k-1}$ holds because the $v_i$s are chosen greedily with ascending degree.
In this case, we actually show that each of the $v_i$s can be matched to a good tight path such that the assigned paths are pairwise vertex-disjoint.
Since the proportion of $k$-sets that are edges in $H'$ is $1 - o(1)$, we know that there are $o(1) \binom{n'}{k-1}$ tuples that are not good in $H'$.
By the result from~\cite{GKL} mentioned in the introduction it holds that
\[ \ex \left(n', \hpath{k-1}{k-2}{2k-2} \right) \leq (k-1) \binom{n'}{k-2} = o(1)\binom{n'}{k-1}.\]
There are at most $O(1) \binom{n-2}{k-2}$ edges including at least two vertices from $V''$.
We assign iteratively vertex-disjoint good $(2k-1)$-paths to each of the $v_i$, $1\leq i \leq t$,
 such that $v_i$ is in each of its edges.
 This is possible, since we disregard at most $o(1) \binom{n'}{k-1}$ many edges in the link of each $v_i$ that contain a tuple that is not good or a vertex contained in a previously assigned path or another vertex from $V''$.
\end{case}
\end{proof}

\begin{proof}[Proof of Theorem~\ref{delta1}]
The proof of Theorem~\ref{delta1} follows the same pattern as the proof of Theorem~\ref{l-tight} without making use of Claim~\ref{match}.
Therefore, we only give a brief sketch of it.

Suppose $H$ is a $k$-graph on $n$ vertices, $n$ sufficiently large, with $\delta_1 \geq (1-\eps)\binom{n-1}{k-1}$ and  $\eps = \frac{1}{22(1280k^3)^{k-1}}$.
Similarly to the proof of Theorem~\ref{l-tight}, we apply Lemmas~\ref{abs} and \ref{con} and obtain via deletion of elements of $\mathcal{A}$ two vertex-disjoint sets $A$ and $C$ such that $\mathcal{A}$ and $\mathcal{C}$ have the desired properties. Using Lemma~\ref{onepath} we find a good path containing all elements of $\mathcal{A}$ such that $\mathcal{C}$ is vertex-disjoint from it. We extend this path with Lemma~\ref{extend} such that it contains all but at most $\rho k n$ vertices from $V\backslash C$ and no vertex from $C$. Using a connector, we connect the ends of this path, obtaining a cycle. As $2 \beta n + k \rho n < \frac{\gamma n}{5}$ holds, we absorb the remaining vertices and obtain a tight Hamiltonian cycle.
\end{proof}

\section{Concluding Remarks}

The edge-density of extremal non-Hamiltonian hypergraphs is $1-o(1)$ (unlike the density of $F$-extremal graphs for fixed $k$-graphs $F$), since a Hamiltonian cycle is a spanning substructure. In general, we conjecture that an extremal graph of any bounded spanning structure consists of an $(n-1)$-clique and a further extremal graph.

\begin{conj}
\label{general}
For any $k \in \mathbb{N}$ there exists an $n_0$ such that for every $k$-graph $H$ on $n \geq n_0$ vertices
without a spanning subgraph isomorphic to a forbidden hypergraph $F$
of bounded maximum vertex degree,
\[|e(H)|\leq \binom{n-1}{ k} + \ex \left(n-1, \left\{F(v):v \in V \right\} \right)\]
holds, and the bound is tight.
\end{conj}

It is not hard to see that Conjecture~\ref{general} holds for forbidden spanning subgraphs $F$ containing a fixed vertex set $V'\subset V$, where $|V'|$ has constant size, such that $F[V\backslash V']$ is a subgraph consisting of a vertex-disjoint union of paths.
Note that  Conjecture~\ref{general} implies the result obtained by the proof of Theorem~\ref{l-tight}.

A $2$-graph is called \emph{pancyclic}, if for any $c$ with $3 \leq c \leq n$ it contains a $c$-cycle.
Similarly to the spanning structure of Hamiltonian $l$-tight cycles,
Katona and Kierstead \cite{KK} defined $l$-tight cycles of any length.
This allows us to generalize the concept of pancyclicity by calling a $k$-graph \emph{$l$-pancyclic},
if for any $c$ with $3 \leq c \leq n/(k-l)$ it contains an $l$-tight cycle on $c$ edges.
In his famous metaconjecture \cite{Bondy}, Bondy claimed for 2-graphs that almost any non-trivial condition on a graph which implies that the graph is Hamiltonian
also implies that the graph is pancyclic. (There may be a simple family of exceptional graphs.)

It is not hard to see that both the condition in Theorem~\ref{l-tight} and the condition in Theorem~\ref{delta1}
imply not only Hamiltonicity but also pancyclicity.

\bibliographystyle{plain}
\bibliography{quellen}

\end{document}